\documentclass[12pt,reqno]{amsart}
\usepackage{hyperref}
\usepackage{amsmath,amssymb}
\usepackage[english]{babel}
\usepackage{caption}
\usepackage{amssymb,amsmath,euscript,enumerate,tikz}
\usepackage[margin=1in]{geometry}
\usepackage{graphicx, caption}
\usepackage{float}
\usepackage{pgf,tikz}
\usepackage{mathrsfs}
\usepackage{upgreek}
\usepackage{mathtools}
\usepackage[utf8]{inputenc}

\newcommand\oc{\operatorname{oc}}

\newcommand \reg{\operatorname{reg}}
\newcommand \Tor{\operatorname{Tor}}

\newcommand \ini{\operatorname{in}}
\newcommand \lex{\operatorname{lex}}
\newcommand \iv{\operatorname{iv}}

\newcommand \I{\mathcal{I}}

\newcommand{\chara}{\operatorname{char}}

\newcommand{\K}{\mathbb{K}}

\newtheorem{theorem}{Theorem}[section]
\newtheorem{definition}[theorem]{Definition}
\newtheorem{lemma}[theorem]{Lemma}
\newtheorem{proposition}[theorem]{Proposition}
\newtheorem{example}[theorem]{Example}
\newtheorem{obs}[theorem]{Observation}
\newtheorem{question}[theorem]{Question}
\newtheorem{remark}[theorem]{Remark}
\newtheorem{corollary}[theorem]{Corollary}

\newtheorem{notation}[theorem]{Notation}

\begin{document}
	\title[Regularity of powers of quadratic sequences]{Regularity of powers of quadratic sequences with applications to
		binomial ideals}
	\author[A. V. Jayanthan]{A. V. Jayanthan}
	\email{jayanav@iitm.ac.in}
	\author[Arvind Kumar]{Arvind Kumar}
	\email{arvkumar11@gmail.com}
	\author[Rajib Sarkar]{Rajib Sarkar}
	\email{rajib.sarkar63@gmail.com}
	\address{Department of Mathematics, Indian Institute of Technology
		Madras, Chennai, INDIA - 600036}

	\begin{abstract}
In this article, we obtain an upper bound for the Castelnuovo-Mumford
regularity of powers of an ideal generated by a homogeneous
quadratic sequence in a polynomial ring in terms of the
regularity of its related ideals and degrees of its
generators. As a consequence, we compute upper bounds for the
regularity of powers of several binomial ideals. We generalize
a result of Matsuda and Murai to show that the regularity of
$J^s_G$ is bounded below by $2s+\ell(G)-1$ for all $s \geq 1$,
where $J_G$ denotes the binomial edge ideal of a graph $G$ and
$\ell(G)$ is the length of a longest induced path in $G$. We
compute the regularity of powers of binomial edge ideals of
cycle graphs, star graphs, and balloon graphs explicitly.
Also, we give sharp bounds for the regularity of powers of
almost complete intersection binomial edge ideals and parity
binomial edge ideals.
\end{abstract}
\keywords{Quadratic sequence, Related ideal, Binomial edge ideal, Castelnuovo-Mumford regularity, $d$-sequence}
\thanks{AMS Subject Classification (2010): 13D02,13C13, 05E40}
\maketitle
\section{Introduction}
Huneke introduced the notion of $d$-sequence, in \cite{Hu82}, and
	proved that the symmetric algebra and Rees algebra of an ideal generated
	by a $d$-sequence in a Noetherian ring are isomorphic, \cite{Hu80} (see
	\cite{KNR91} for a simple proof). He used
	the theory of $d$-sequence to study the depth of powers of ideals in
	a Noetherian ring $R$, \cite{Hu82}. He generalized this notion to weak
	$d$-sequence and analyzed behavior of $R/I^n$, when $I$ is generated
	by a weak $d$-sequence, \cite{Hu81}. Raghavan further generalized the notion of weak $d$-sequence to quadratic sequence and
	studied the depth of $R/I^n$ when $I$ is generated by a quadratic
	sequence, \cite{KNR}. In this paper, we obtain an upper bound for the
	Castelnuovo-Mumford regularity of $R/I^n$ in terms of the regularity of related ideals, and
	degrees of the generators, where $R$ is a standard graded polynomial
	ring over a field $\K$ and $I$ is an ideal generated by a homogeneous 
	quadratic sequence,  Theorem \ref{reg-graded-filtration},
	Corollary \ref{reg-d-sequence}. To illustrate our result, we compute
	the regularity of powers of the defining ideal of a class of projective
	monomial curves, which in fact, is a binomial ideal generated by a
	quadratic sequence.
	
	Ever since Cutkosky, Herzog and Trung, in \cite{CHT}, and independently
	Kodiyalam, in \cite{Kod}, proved that if $I$ is a homogeneous ideal in
	a polynomial ring $R$, then $\reg(I^s) = as + b$ for $s \gg 0$, for
	some non-negative integers $a, b$, it has been a constant effort from
	the researchers to compute $a$ and $b$ for several classes of
	homogeneous ideals. They showed that $a$ is at most the maximum degree
	of a minimal homogeneous generator of $I$. It has remained a challenge
	to compute the constant term in the linear polynomial. During the past decade, there has been a lot of research activity in this direction.
	In particular, if $I(G)$ denotes the monomial edge ideal corresponding
	to a finite simple graph $G$, then researchers have obtained an
	upper bound for the constant term for all graphs and have computed
	the constant term for several subclasses of graphs
	(see \cite{BBH17, JS18} and the references therein). 
	While monomial ideals have received a lot of attention in this
	direction, there are not many such results for binomial ideals.
	Recently, Raicu computed the linear polynomial corresponding to the
	asymptotic regularity function for $p\times p$ minors of an $m \times
	n$ matrix, $m \geq n$, \cite{Rai18}.  As an application of our result,
	we get upper bounds for the regularity of powers of certain binomial
	ideals, namely, binomial edge ideals and parity binomial edge ideals.
	
	Let $G$ be a simple graph with the vertex set $[n]=\{1,\dots,n\}$ and
	the edge set $E(G)$. Let $S=\K[x_1,\dots,x_n,y_1,\dots,y_n]$ be the
	polynomial ring where $\K$ is an arbitrary field.  The binomial edge
	ideal corresponding to $G$, denoted by $J_G$, is the ideal generated
	by the set $\{x_iy_j-x_jy_i : i<j \text{ and } \{i,j\}\in E(G)\}$. The
	notion of binomial edge ideal was introduced by Herzog et al. in
	\cite{HH1} and independently by Ohtani in \cite{oh}. Another class of
	binomial ideals associated with finite simple graphs are the parity
	binomial edge ideals. For a graph $G$ on the vertex set $[n]$, the
	parity binomial edge ideal $\I_G$ is the ideal generated by the set
	$\{x_ix_j - y_iy_j : \{i, j\}\in E(G), i < j \} \subset S$. Kahle et
	al. introduced this notion, \cite{kah}, and studied its various
	properties.  
	
In the recent past, researchers have been trying to understand various
properties of these ideals and their relationship with combinatorial
properties of corresponding graphs. While there is some success in the
case of binomial edge ideals, the parity binomial edge ideals are
quite new and nothing much is known about them. One line of research
is to estimate the regularity of these ideals using combinatorial
invariants of corresponding graphs. In \cite{MM}, Matsuda and Murai
proved that $\ell(G) \leq \reg(S/J_G) \leq n-1$, where $\ell(G)$
denotes the length of a longest induced path in $G$. This bound, in
general, is a weak one and there are improved bounds for several
classes, (see for example
\cite{her2,JNR2,KMJCTA,AR3,AR2,RCJAA,MKM2018}). For some classes of
graphs, precise expressions for the regularity have also been
computed, (see for example \cite{EZ, JA1, Schenzel, Zafar}).  The
lower bound in the Matsuda-Murai bound for the regularity was a
consequence of a more general result, namely, if $H$ is an induced
subgraph of $G$, then $\beta_{i,j}(S/J_H) \leq \beta_{i,j}(S/J_G)$ for
all $i, j$. We generalize this result to all powers, that is,
$\beta_{i,j}(S/J_H^s) \leq \beta_{i,j}(S/J_G^s)$ for all $i, j$ and $s
\geq 1$, whenever $H$ is an induced subgraph of $G$, Proposition
\ref{induced-subgraph}. As an immediate consequence, we obtain a
general lower bound, namely, $2s + \ell(G) - 2 \leq \reg(S/J_G^s)$ for
all $s \geq 1$, Corollary \ref{lower-bound}.
	
Computing the regularity of powers of (parity) binomial edge ideals of
an arbitrary graph seems more challenging compared to the regularity
of powers of monomial edge ideals. Even in the case of simple classes
of graphs, the regularity of the powers of their binomial edge ideals
is not known. So, naturally one restricts the attention to important
subclasses.  In \cite{JAR1}, we studied the Rees algebra and first
graded Betti numbers of binomial edge ideals which are almost complete
intersections. We proved that almost complete intersection binomial
edge ideals are generated by $d$-sequence. Cutkosky, Herzog and Trung
proved that if $I$ is an ideal generated by a $d$-sequence of $n$
forms of the same degree $r$, then for all $s \geq n+1,$ $\reg(I^s) =
(s-n-1)r + \reg(I^{n+1})$, \cite[Corollary 3.8]{CHT}. This expression
depends on the number of generators of $I$. Moreover, in our
situation, computing the linear polynomial boils down to computing
$\reg(I^{n+1})$, which itself is challenging when the graph has a
large number of edges.  Note that a $d$-sequence is a quadratic
sequence. Moreover, in the case of ideals generated by $d$-sequence,
the computation of the related ideals becomes much simpler. Using the
upper bounds in Theorem \ref{reg-graded-filtration} and Corollary
\ref{reg-d-sequence}, we compute the regularity of powers of binomial
edge ideals of cycles, star graphs and balloon graphs. For other
almost complete intersection binomial edge ideals, we obtain bounds
for the regularity of their powers. 
	
\begin{theorem} Let $G$ be a finite simple graph and $J_G$ denote its
binomial edge ideal in the polynomial ring $S$.
\begin{enumerate}
	\item If $G = K_{1,n}$, then $\reg(S/J_G^s) = 2s$ for all $s \geq
		1$.
	\item If $G = C_n$, then $\reg(S/J_G^s) = 2s + n - 4$ for all $s
		\geq 1$.
	\item If $G$ is a tree such that $J_G$ is an almost complete
		intersection ideal, then 
		\[2s + \iv(G) - 2 \leq \reg(S/J_G^s) \leq 2s + \iv(G) - 1,\]
		for all $s \geq 1$, where $\iv(G)$ denotes
		the number of internal vertices of $G$.
			
	\item If $G$ is a unicyclic graph on $[n]$ 
		such that $J_G$ is an almost complete intersection ideal, then 
		\[2s + n - 5 \leq \reg(S/J_G^s) \leq 2s + n - 4\]
		for all $s \geq 1$.
\end{enumerate}
	\end{theorem}
	
	Bolognini et al. proved that if $G$ is bipartite, then
	$J_G$ and $\I_G$ are isomorphic, \cite{dav}. Therefore, to study
	parity binomial edge ideals, we consider graphs containing an odd-cycle.
	For parity binomial edge ideals, we prove:
	\begin{theorem} For all $s \geq 1$,
		\begin{enumerate}
			\item if $n \geq 3$ is an odd integer, then $\reg(S/\I_{C_n}) = 2s + n -2$,
			\item if $G$ is a graph on $[n]$ obtained by adding an edge
			between an odd cycle and an internal vertex of a path, then $2s
			+ n -5 \leq \reg(S/\I_G^s) \leq 2s + n - 4$,
			\item if $G$ is either a balloon graph on $[n]$ having odd girth
			or a graph obtained by adding a chord in an odd cycle $C_n$,
			then $2s + n - 4 \leq \reg(S/\I_G^s) \leq 2s + n - 3$.
		\end{enumerate}
		
	\end{theorem}
	
\textbf{Acknowledgement.} The authors are thankful to K. N. Raghavan
for pointing us to several useful references on quadratic sequences
and suggestions on the computation of their related ideals.  The
National Board for Higher Mathematics, India partially supports the
first author through the project, No. 02011/23/2017/R\&D II/4501. The
second author is financially supported by the National Board for
Higher Mathematics, India. The University Grants Commission, India
financially supports the third author. We are also thankful to the
referee for suggesting a simpler proof of Proposition
\ref{clique-sum-regularity} for which we had a slightly lengthier
proof.

\section{Regularity of powers of quadratic sequence}
In this section, we study the regularity of powers of an ideal generated
by quadratic sequence. First, we recall the definition of quadratic
sequence from \cite{KNR}.
	
	Let $\Lambda$ be a finite poset. A subset $\Sigma \subseteq \Lambda$
	is said to be a \textit{poset ideal}
	if it satisfies the following property:
	\[\text{if }\sigma \in \Sigma  \text{ and }\lambda\in \Lambda \text{
		with }\lambda\leq \sigma, \text{ then }\lambda\in \Sigma. \]
	Let $R$ be a commutative Noetherian ring with unity and $
	\{u_{\lambda} : \lambda\in \Lambda \}$ be a set of elements of $R$
	indexed by $\Lambda$. For $\Sigma \subseteq \Lambda$, let $U_{\Sigma}$
	denote the ideal of $R$ generated by $\{u_{\sigma} : \sigma\in \Sigma
	\}$. Note that $U_{\emptyset}=(0 )$. Let $\Sigma$ be a poset ideal of
	$\Lambda$ and $\lambda\in \Lambda$. We say that $\lambda$ \textit{lies
		just above} $\Sigma$ if it satisfies the following:
	\begin{enumerate}[\rm i)]
		\item $\lambda \notin \Sigma$ and
		\item $\sigma\in \Sigma$, whenever $\sigma\in \Lambda$ and $\sigma
		<\lambda$.
	\end{enumerate}
	An element $\lambda\in \Lambda$ is said to \textit{lie inside or just
		above}
	$\Sigma$ if either $\lambda\in \Sigma$ or $\lambda$ lies just above
	$\Sigma$.
	\begin{definition} $($\cite[Definition 3.3]{KNR}$)$\label{qsdef}
		Let $\Lambda$ be a finite poset and $I\subset R$ be an ideal. Let
		``-'' denote images in $R/I$.
		A set of elements $\{u_{\lambda} : \lambda \in \Lambda \}\subseteq R$
		is said to be a quadratic sequence with respect to the ideal $I$ if
		for every pair $(\Sigma,\lambda)$, where $\Sigma$ is a poset ideal of
		$\Lambda$ and $\lambda$ lies inside or just above $\Sigma$, there
		exists a poset ideal $\Theta$ of $\Lambda$ such that 
		\begin{enumerate}
			\item ($\bar{U}_{\Sigma}:\bar{u}_{\lambda})\cap
			\bar{U}_{\Lambda}\subseteq \bar{U}_{\Theta}$,
			\item $u_{\lambda}U_{\Theta} \subseteq (U_{\Sigma}+I)U_{\Lambda}.$
		\end{enumerate}
		A set of elements $\{u_{\lambda}:\lambda \in \Lambda \}\subseteq R$ is
		said to be a quadratic sequence if it is a quadratic sequence with
		respect to the zero ideal. 
	\end{definition}
	We now recall some basic properties of quadratic sequences from
	\cite{KNR} which are required for our results.
	\begin{obs}$($\cite[Remark 3.4]{KNR}$)$
		\begin{enumerate}
			\item Let $I$ be an ideal of $R$. If $\{u_{\lambda}:\lambda
			\in \Lambda \}\subseteq R$ is a quadratic sequence with
			respect to $I$, then $\{\bar{u}_{\lambda}:\lambda \in \Lambda \}
			\subseteq R/I$ is also a quadratic sequence.
			\item If $\{u_{\lambda}:\lambda \in \Lambda \}\subseteq R$ is
			a quadratic sequence, then for any poset ideal $\Sigma
			\subset \Lambda$, $\{\bar{u}_{\lambda}:\lambda \in \Lambda\setminus \Sigma \}\subseteq R/U_{\Sigma}$ is a quadratic sequence.
		\end{enumerate}
	\end{obs}
	\begin{lemma}$($\cite[Corollaries 3.7 and 5.2]{KNR}$)$\label{tech-lemma}
		Let $\{u_{\lambda}:\lambda \in \Lambda \}\subseteq R$ be a quadratic sequence. Then
		\begin{enumerate}
			\item for every poset ideal $\Sigma$ of $\Lambda$,
			$U_{\Sigma}\cap U^s_{\Lambda}=U_{\Sigma}U_{\Lambda}^{s-1}$
			for any integer $s\geq 1$ and
			\item for any minimal element $\alpha$ of $\Lambda$,
			$\{\bar{u}_{\lambda}:\lambda \in \Lambda \}\subseteq R/(0:u_{\alpha})$ is a quadratic sequence.
		\end{enumerate}
	\end{lemma}
\textit{Related ideals} help to understand the structure of the
ideals generated by quadratic sequence and their powers. We recall
its definition here.
	\begin{definition} $($\cite[Definition 5.3]{KNR}$)$ 
		Let $\{u_{\lambda}:\lambda \in \Lambda \}\subseteq R$ be a quadratic
		sequence.  An ideal $J\subseteq R$ is said to be a related ideal to
		the quadratic sequence if $J=U_{\Lambda}$ or
		$J=(U_{\Sigma}:u_{\lambda})+U_{\Lambda}$ for some pair
		$(\Sigma,\lambda)$, where $\Sigma$ is a poset ideal of $\Lambda$ and
		$\lambda$ lies inside or just above $\Sigma$.
	\end{definition}
	
	In the following, we separate out a result from the proof
	of Theorem 5.4 of \cite{KNR} which is required for the main theorem in
	this section.
	
	\begin{lemma}\label{tech-lemma1}
		Let $\{u_{\lambda}:\lambda \in \Lambda \}\subseteq R$ be a quadratic
		sequence and $\alpha$ be a minimal element of $\Lambda$. Let $\Sigma$
		be a poset ideal of $\Lambda$ and $\lambda\in \Lambda$ lies inside or
		just above $\Sigma$.  Then
		$((U_{\Sigma}+(0:u_{\alpha})):u_{\lambda})+U_{\Lambda}$ is a related
		ideal to the quadratic sequence $\{u_{\lambda} : \lambda\in \Lambda
		\}$.
	\end{lemma}
	One of the important aspects in the study of powers of ideals generated
	by quadratic sequence is the existence of a filtration with some nice properties.  We now prove a graded version of this result,
	\cite[Theorem 5.4]{KNR}. The proof is similar to that of the original
	result. We include it here for the sake of completeness.
	
	\begin{theorem}\label{graded-filtration}
		Let $R = \oplus_{n\geq 0}R_n$ be a graded $R_0$-algebra, where $R_0$
		is a Noetherian ring. Let $\Lambda$
		be a finite poset and 
		$\{u_{\lambda}:\lambda \in \Lambda \}\subseteq R$ be a set of
		homogeneous elements of $R$ with
		$\deg(u_{\lambda})=d_{\lambda}>0$. Set $d=\max\{
		d_{\lambda}: {\lambda \in \Lambda}\}$.  If $\{u_{\lambda}:\lambda \in \Lambda \}$ is a
		quadratic sequence, then for every $s\geq 1$, there exists a graded
		filtration of $R/U_{\Lambda}^s$ 
		\[R/U_{\Lambda}^s=M_0\supseteq M_1\supseteq \cdots
		\supseteq M_k=(0)\]
		such that for every $0\leq i\leq k-1$, there
		exists a related ideal $V_i$ and $0\leq d_i\leq d(s-1)$ with 
		$M_i/M_{i+1}\simeq [R/V_i](-d_i)$. 
	\end{theorem}
	
	\begin{proof}
		We prove the assertion by induction on $|\Lambda|+s$. If $s=1$, then
		$(0)\subseteq R/U_{\Lambda}$ is the required filtration. Assume that $s\geq
		2$. Let $|\Lambda|=1$. Set $\Lambda=\{\alpha\}$. Then, for any $s \geq 2$,
		\begin{eqnarray}\label{filtration1}
			R/(u^s_{\alpha})\supseteq (u_{\alpha})/(u^s_{\alpha})\supseteq
			\dots \supseteq
			(u_\alpha^{s-1})/(u^s_{\alpha})\supseteq (0)
		\end{eqnarray}
		is a filtration of $R/(u^s_{\alpha})$.
		Take $\Sigma = \emptyset$, then $\alpha$ lies just above $\Sigma$. Hence, \ref{qsdef}(1) and (2) translates to
		$(0 : u_\alpha) \cap (u_\alpha) \subseteq U_\Theta$ and $u_\alpha U_\Theta =
		(0)$ for some poset ideal $\Theta$. If $\Theta = \{\alpha\}$,
		then $u_\alpha^2 = 0$. Therefore, $u_\alpha \in (0 : u_\alpha)$, and
		hence, $R/(0: u_\alpha) = R/( (0 : u_\alpha)+(u_\alpha)) \cong
		(u_\alpha)$. Since both $(u_\alpha)$ and $(0 : u_\alpha) + (u_\alpha)$
		are related ideals, $R \supset (u_\alpha) \supset (0)$ is the
		required filtration. Now, suppose $\Theta = \emptyset$. Then, $(0 : u_\alpha) \cap (u_\alpha)
		= (0)$. Let $k \geq 2$ and $a
		u_\alpha^k = 0$. Then, $a u_\alpha^{k-1} \in (0 : u_\alpha) \cap
		(u_\alpha) = 0$. Hence, $a \in (0 : u_\alpha^{k-1})$. Therefore, $(0 :
		u_\alpha^k) = (0 : u_\alpha^{k-1})$.
		
		Now, we show that $\frac{(u^k_{\alpha})}{(u^{k+1}_{\alpha})}\simeq
		\frac{(u^{k-1}_{\alpha})}{(u^k_{\alpha})}(-d_\alpha)$ for
		$k\geq 2$. Consider $\mu_{u_\alpha} :  (u_\alpha^{k-1}) \rightarrow
		\frac{(u^k_{\alpha})}{(u^{k+1}_{\alpha})},$ the multiplication by
		$u_\alpha$.
		Let $a \in (u_\alpha^{k-1})$ be such that $au_\alpha \in
		(u_\alpha^{k+1})$. Write $a = fu_{\alpha}^{k-1}$ and 
		$au_{\alpha} =gu_{\alpha}^{k+1}$ for some $f,g\in R$. 
		Then, for $k \geq 2$, $(f-gu_{\alpha})\in (0:u_{\alpha}^k)=(0:u_{\alpha}^{k-1})$ and so
		$fu_{\alpha}^{k-1}\in (u_{\alpha}^{k})$. Therefore, for $k\geq 2$, 
		\[
		\frac{(u^k_{\alpha})}{(u^{k+1}_{\alpha})}\simeq
		\frac{(u^{k-1}_{\alpha})}{(u^k_{\alpha})}(-d_{\alpha})\simeq
		\dots \simeq \frac{(u_{\alpha
			})}{(u^2_{\alpha})}(-(k-1)d_{\alpha})\simeq
		\frac{R}{((0:u_{\alpha})+(u_{\alpha}))}(-kd_{\alpha}),\]
		where the last isomorphism is obtained by proving that the kernel of
		the multiplication mapping from $R$ to $(u_{\alpha})/(u_{\alpha}^2)$
		is $((0:u_{\alpha})+(u_{\alpha})).$ Note that
		$(0:u_{\alpha})+(u_{\alpha})$ is a related ideal, and hence, the 
		filtration (\ref{filtration1}) satisfies the required conditions. This
		completes the case $|\Lambda| = 1$.
		
		Now, assume that $|\Lambda|\geq
		2$ and $s\geq 2$. Let $\alpha\in \Lambda$ be a minimal element.
		Consider the filtration
		$R/U_{\Lambda}^s\supseteq (u_{\alpha},U_{\Lambda}^s)/U_{\Lambda}^s
		\supseteq (0).$ 
		It follows from Lemma \ref{tech-lemma}(1) that
		$\frac{(u_{\alpha},U_{\Lambda}^s)}{(U_{\Lambda}^s)}\simeq
		\frac{(u_{\alpha})}{(u_{\alpha})\cap U_{\Lambda}^s}\simeq
		\frac{(u_{\alpha})}{u_{\alpha}U_{\Lambda}^{s-1}}.$ It is easy to see that the
		kernel of the multiplication map from $R$ to
		$\frac{(u_{\alpha})}{u_{\alpha}U_{\Lambda}^{s-1}}$ is $(0:u_{\alpha})+U_{\Lambda}^{s-1}$.
		Therefore, 
		$\frac{R}{(0:u_{\alpha})+U_{\Lambda}^{s-1}}(-d_{\alpha})\simeq \frac{(u_{\alpha})}{u_{\alpha}U_{\Lambda}^{s-1}}.$
		Set $\bar{R}=R/(u_{\alpha})$ and $ \Lambda' = \Lambda \setminus \{\alpha\}$. 
		Since $\{\bar{u}_{\lambda}:\lambda\in \Lambda' \}\subseteq
		R/(u_{\alpha})$ is a quadratic sequence and $|\Lambda' |<|\Lambda|$,
		by induction $\bar{R}/\bar{U}^s_{\Lambda'}$ has a graded filtration 
		\begin{eqnarray}\label{filtration2}
			\bar{R}/\bar{U}^s_{\Lambda'}=N_0\supseteq N_1\supseteq \dots
			\supseteq N_l=(0)
		\end{eqnarray}
		of $\bar{R}$-modules such that for each $0\leq
		j\leq l-1$, there exists $\bar{V}_j$, a related ideal to the quadratic
		sequence $\{\bar{u}_{\lambda}:\lambda\in \Lambda' \}$ and $0\leq
		d_j\leq d(s-1)$ such that $N_j/N_{j+1}\simeq
		[\bar{R}/\bar{V_j}](-d_j)$. 
		If $\bar{V_j}= \bar{U}_{\Lambda'}$, then the pre-image of $\bar{V}_j$ in
		$R$ is $U_{\Lambda}$, and therefore, $N_j/N_{j+1}\simeq [R/U_{\Lambda}](-d_j).$     So, assume that     $\bar{V_j}=(\bar{U}_{\Sigma_j}:\bar{u}_{\lambda_j})+\bar{U}_{\Lambda'}$
		for some poset ideal $\Sigma_j$ of $\Lambda'$ and $\lambda_j\in \Lambda'$
		lies inside or just above $\Sigma_j$.  The pre-image of $\bar{V}_j$ in
		$R$ is $(U_{\Sigma_j \cup \{\alpha\}}:u_{\lambda_j})+U_{\Lambda}$ so that
		$$N_j/N_{j+1}\simeq [R/((U_{\Sigma_j \cup
			\{\alpha\}}:u_{\lambda_j})+U_{\Lambda})](-d_j).$$ Since $\alpha$ is a
		minimal element in $\Lambda$ and $\Sigma_j$ is a poset ideal of
		$\Lambda'$, $\Sigma_j \cup
		\{\alpha\}$ is a poset ideal of $\Lambda$ and $\lambda_j$ lies inside
		or just above $\Sigma_j \cup \{\alpha\}$. Therefore, $(U_{\Sigma_j \cup
			\{\alpha\}}:u_{\lambda_j})+U_{\Lambda}$ is a related ideal to the
		quadratic sequence $\{u_\lambda ~ : ~ \lambda \in \Lambda\}$. Hence, 
		$R/(u_{\alpha},U_{\Lambda}^s)$ has the required graded filtration. Note that by
		Lemma \ref{tech-lemma}(2), $\{u'_{\lambda}:\lambda \in \Lambda
		\}\subseteq R/(0:u_{\alpha})$ is a quadratic sequence, where $'$
		denotes the image modulo the ideal $(0:u_{\alpha})$.
		Thus, by induction, there exists a graded filtration 
		\begin{eqnarray}\label{filtration3}
			R'/{U'}_{\Lambda}^{s-1}=
			L_0\supseteq L_1\supseteq \dots \supseteq L_k=(0)
		\end{eqnarray}
		of $R'$-modules
		such that for each $0\leq i\leq k-1$, there exists a related ideal to
		$\{u'_{\lambda} : \lambda \in \Lambda\}$, $V_i'\subset R'$ and $0\leq d'_i\leq d(s-2)$ such that
		$L_i/L_{i+1}\simeq [R'/V'_i](-d'_i).$ 
		Since $V'_i$'s are related
		ideals to $\{u'_{\lambda}:\lambda \in \Lambda \}$, $V'_i$'s are either $U_{\Lambda}'$  or  has the
		form $(U'_{\Sigma_i}:u'_{\lambda_i})+U'_{\Lambda}$ for some poset
		ideal $\Sigma_i$ of $\Lambda$ and $\lambda_i$ lying inside or just above
		$\Sigma_i$. Hence, $[R'/V_i'](-d_i')\simeq
		[R/(((U_{\Sigma}+(0:u_{\alpha})):u_{\lambda})+U_{\Lambda})](-d_i')$ or $[R'/V_i'](-d_i')\simeq [R/((0:u_{\alpha})+U_{\Lambda})](-d_i').$
		By Lemma \ref{tech-lemma1},
		$((U_{\Sigma}+(0:u_{\alpha})):u_{\lambda})+U_{\Lambda}$ is a related
		ideal to $\{u_{\lambda}:\lambda\in \Lambda \}$. Hence, we get a
		graded filtration for $[R'/{U'}_{\Lambda}^{s-1}](-d_\alpha) \simeq (u_{\alpha},U_{\Lambda}^s)/U_{\Lambda}^s$.
		By combining the filtrations (\ref{filtration2}) and
		(\ref{filtration3}) we get the required filtration of $R/U_{\Lambda}^s$.
	\end{proof}
	
	The following result on the regularity is well-known and can be
	derived from the long exact sequence of Tor modules. We state it for
	the sake of convenience.
	\begin{lemma}\label{regularity-lemma}
		Let $ R $ be a standard graded ring and $M,N, P $ be finitely generated graded $ R $-modules. 
		If $ 0 \rightarrow M \xrightarrow{f}  N \xrightarrow{g} P \rightarrow 0$ is a 
		short exact sequence with $f,g$  
		graded homomorphisms of degree zero, then 
		\begin{enumerate}
			\item $\reg (N) \leq \max \{\reg (M), \reg (P)\},$
			\item $\reg(M)\leq \max \{\reg(N),\reg(P)+1 \}$,
			\item $\reg(P)\leq \max \{\reg(M)-1,\reg(N)\}$,
			\item $\reg (M) = \reg (P)+1$ if  $\reg (N) < \reg (M)$.
		\end{enumerate}	
	\end{lemma}	
	
	We now prove an upper bound for the regularity of powers
	of an ideal generated by a homogeneous quadratic sequence.
\begin{theorem}\label{reg-graded-filtration}
Let $R $ be a standard graded polynomial ring over a field $\K$. Let
$\Lambda$ be a finite poset and  $\{u_{\lambda}:\lambda \in
\Lambda\}\subseteq R$ be a quadratic sequence. Then for $s \geq 1$
$$\reg(R/U_{\Lambda}^s)\leq d(s-1)+\max_{\Sigma,\lambda} \reg(R/(
(U_{\Sigma}:u_{\lambda})+U_{\Lambda})),$$ where $\Sigma$ is a poset
ideal of $\Lambda$ and $\lambda$ lies inside or just above $\Sigma$,
and $d=\max\{ \deg(u_{\lambda}) : {\lambda \in \Lambda}\}$.
\end{theorem}
\begin{proof}
By Theorem \ref{graded-filtration}, there exists a graded filtration
$$R/U_{\Lambda}^s=M_0\supseteq M_1\supseteq \cdots \supseteq
M_k=(0)$$ such that for every $0\leq i\leq k-1$, there exists a
related ideal $V_i$ and $0\leq d_i\leq d(s-1)$ with $M_i/M_{i+1}\simeq
[R/V_i](-d_i)$.  For $0 \leq i \leq k-1$, consider the following short
exact sequence: $$0 \longrightarrow M_{i+1} \longrightarrow M_i
\longrightarrow M_i/M_{i+1} \longrightarrow 0.$$  By applying Lemma
\ref{regularity-lemma} successively in above short exact sequences, we
get 
\begin{eqnarray*}
	\reg(R/U_{\Lambda}^s)& \leq & \max \{\reg(M_{i}/M_{i+1})
	: 1 \leq i \leq k-1\} \\ & = &\max \{d_i+\reg(R/V_i) : 1\leq i \leq
	k-1\} \\ & \leq & d(s-1)+\max_{\Sigma,\lambda} \reg(R/(
	(U_{\Sigma}:u_{\lambda})+U_{\Lambda})), 
\end{eqnarray*}
where $\Sigma$ is a poset ideal of $\Lambda$ and $\lambda$ lies inside or just above $\Sigma$.	
\end{proof}
We now illustrate the use of the above theorem with an example. This
example also shows that the upper bound that we have obtained is sharp.
\begin{example}{\em
Let $R=\mathbb{K}[x,y,z,w]$ be a polynomial ring and $U$ denote the defining ideal for the projective monomial curve 
\[ (x:y:z:w)=(u^{b+c},u^b v^c,u^c v^b,v^{b+c}), \]
with $\gcd(b,c)=1$ and $b>c$.  Morales and Simis proved,
\cite[Proposition 2.2]{ms92}, that $U$ is minimally generated by
$b-c+2$ elements, which are 
\[ u_1=xw-yz, u_2=x^{b-c}z^c-y^b, \dots, u_{b-c+1}=
xz^{b-1}-w^{b-c-1}y^{c+1}, u_{b-c+2}=z^b-w^{b-c}y^c.  \]
They proved that $u_1, \ldots, u_{b-c+2}$ is a quadratic sequence (even a weak
$d$-sequence). Note that $U$ is not
generated by a $d$-sequence except when $b=3$ and $c=2$. 
Let $U_i=((u_1,\dots,u_{i-1}):u_i)$. Then, from the proof of
Proposition 2.2 of \cite{ms92}, we get $U_1=(0)$, $U_2=(xw-yz)$, 
$U_3=\dots=U_{b-c+2}=(x,y)$. Hence, the related ideals to the
quadratic sequence $\{u_1,\dots,u_{b-c+2}\}$ are either $U$ or $U_i+U = (x, y,
z^b)$ for $3 \leq i \leq b-c+2$.
\vskip 2mm \noindent
\textbf{Claim:} For all $s \geq 1$, $\reg(R/U^s)=bs-1$.\\
\textit{Proof.}
First we prove that $\reg(R/U)=b-1$. 
In \cite[Proposition 2.1]{ms92}, Morales and Simis computed a graded
minimal $R$-presentation of $U$, which is as follows: \[
R(-(b+1))^{2(b-c)}\longrightarrow R(-2)\oplus
R(-b)^{b-c+1}\longrightarrow U\longrightarrow 0. \] This implies
that $\beta_{1,1+j}(R/U) = 0 = \beta_{2,2+j}(R/U)$ for $j\geq b$.
Note that $(u_1, u_2)$ is a
regular sequence and for $2\leq k\leq b-c+1$, we consider the
following short exact sequence:
\[ 0 \longrightarrow [R/(x,y)](-b)\stackrel{\cdot
	u_{k+1}}{\longrightarrow} R/(u_1,\dots,u_k) \longrightarrow
R/(u_1,\dots,u_{k+1})\longrightarrow 0. \]
Since $(u_1, u_2)$ is a regular sequence, $\reg(R/(u_1,u_2)) = b$.
Applying Lemma \ref{regularity-lemma} on the above short exact sequence, we get $\reg(R/(u_1,u_2,u_3))\leq b$, \dots, $\reg(R/U)\leq b$.
It can be seen that $\beta_{2,2+b}([R/(x,y)](-b))=1$ and
$\beta_{2,2+b}(R/(u_1,u_2))=1$ are the unique extremal Betti numbers of
$[R/(x,y)](-b)$ and $R/(u_1,u_2)$ respectively. Since
$\beta_{2,2+j+1}([R/(x,y)](-b))=0$ for $j\geq b$,
we have the corresponding long exact sequence of Tor for $2\leq k\leq b-c+1$:
\[0\rightarrow \Tor_{3,3+j}\left(\frac{R}{(u_1,\dots,u_k)}
\right)\rightarrow \Tor_{3,3+j}\left(\frac{R}{(u_1,\dots,u_{k+1})}
\right)\rightarrow \Tor_{2,3+j}\left(\frac{R}{(x,y)}(-b)
\right)\rightarrow \cdots
\]
Thus, for $2\leq k\leq b-c+1$ and $j\geq b$,
$\beta_{3,3+j}(R/(u_1,\dots,u_k))=\beta_{3,3+j}(R/(u_1,\dots,u_{k+1}))
= 0 $. Also we have for $2\leq k\leq b-c+1$,
$\beta_{i,i+j}(R/(u_1,\dots,u_{k+1}))=0$ for $i\geq 4$. Hence, $\reg(R/U)=b-1$.

Note that $U+ U_i = U$ for $i = 1, 2$ and for $3 \leq i \leq b+c-2$, $\reg(R/(U + U_i)) =
\reg(R/(x,y,z^b)) = b - 1$.  Therefore, by Theorem \ref{reg-graded-filtration},
we have $\reg(R/U^s)\leq bs-1$. Since there is an
element of degree $bs$ in $U^s,$ $bs-1\leq \reg(R/U^s)$, and hence,
$\reg(R/U^s)=bs-1$.  } \qed
	\end{example}
	
It follows from Theorem \ref{reg-graded-filtration} that given an
ideal $U_\Lambda$ generated by a quadratic sequence, an upper bound
for the regularity of all its powers can be obtained once we know the
regularity of its related ideals. Given the ideal $U_\Lambda$ and the
poset structure of $\Lambda$, one can compute the related ideals and
their regularity. But, in general, structure of the related
ideals is not very well-understood. Here we study ideals generated by
$d$-sequence for which the related ideals structure is much simpler.

Let $u_1,\dots,u_n$ be homogeneous elements in $R$.
Then, ${u_1},\dots,{u_n}$ is said to be a \textit{homogeneous $d$-sequence} if
\begin{enumerate}
\item $u_i$ is not in the ideal generated by the rest of the
$u_j$'s and
\item for all $k \geq i+1$ and all $i \geq 0,$ $( (u_1, \ldots, u_i)
: u_{i+1}u_k) = ( (u_1, \ldots, u_i) : u_k)$.
\end{enumerate}
Costa, in \cite{costa85}, proved that $(u_1, \ldots, u_n)$ is a $d$-sequence if and only
if for $1 \leq i \leq n$  
\[((u_1,\dots,u_{i-1}):u_i) \cap (u_1,\ldots,u_n)=(u_1,\dots,u_{i-1}).
\]
	
It is clear that if $u_1,\ldots,u_n$ is a $d$-sequence, then they form
a quadratic sequence with respect to the poset $\{1<\cdots <n\}$.  Let
$u_1,\dots,u_n$ be a homogeneous $d$-sequence in $R$ such that
$u_1,\dots,u_{n-1}$ is a regular sequence. Then,
one can note that the related ideals to
$\{u_1,\dots,u_n\}$ in $R$ are of the form $(u_1,\dots,u_n)$ or
$((u_1,\dots,u_{n-1}):u_n)+(u_n)$. First we obtain an upper bound for
the regularity of these related ideals.
	
\begin{proposition}\label{reg-relative-sequence}
Let $u_1,\dots,u_n$ be a homogeneous $d$-sequence with $\deg(u_i)=d_i$
in a standard graded polynomial ring $R$ over a field $\K$ such
that $u_1,\dots,u_{n-1}$ is a regular sequence.  Then,  
\[\reg({R}/{(((u_1,\dots,u_{n-1}):u_n),u_n)})\leq \max \{\reg({R}/{(u_1,\dots,u_n)}),\sum_{i=1}^{n-1} d_i-n \} .\]
\end{proposition}
\begin{proof} For convenience, let $U = (u_1, \ldots, u_n)$ and $U' =
(u_1, \ldots, u_{n-1})$.
Consider the following short exact sequence:
\[
0\longrightarrow \frac{R}{(U':u_n)}(-d_n)\stackrel{\cdot
	u_n}{\longrightarrow} \frac{R}{U'}\longrightarrow\frac{R}{U}\longrightarrow 0.
\]
Since $u_1,\dots,u_{n-1}$ is a regular sequence with $\deg(u_i)  = d_i,$
\[\reg ({R}/U' )=\sum_{i=1}^{n-1}(d_i-1)=\sum_{i=1}^{n-1}d_i -(n-1).\] 
Therefore, by Lemma \ref{regularity-lemma}, 
\[\reg({R}/{(U':u_n)})+d_n\leq \max \{\reg({R}/{U})+1, \sum_{i=1}^{n-1}d_i - (n-1) \}.\]

Now, consider the following short exact sequence:
\begin{align*}
0\longrightarrow \frac{R}{(U':u^2_n)}(-d_n)\stackrel{\cdot
u_n}{\longrightarrow} \frac{R}{(U':u_n)}\longrightarrow\frac{R}{((U':u_n),u_n)}\longrightarrow 0.
\end{align*}
Since $u_1,\dots,u_n$ is  a $d$-sequence, $(U':u^2_n)=(U':u_n)$.
Therefore, it follows from Lemma \ref{regularity-lemma} that
\begin{eqnarray*}
\reg ( {R}/{((U':u_n),u_n)})& = & \reg ({R}/{(U':u_n)} )+d_n-1\\			
& \leq & \max \left\{ \reg({R}/{U}), \sum_{i=1}^{n-1}d_i -
n\right \}
\end{eqnarray*}
and this completes the proof.
\end{proof}
	
As an immediate consequence, we obtain an upper bound for the
regularity of powers of this class of ideals.
\begin{corollary}\label{reg-d-sequence}
Let $R$ be a standard graded polynomial ring over a field $\K$ and
$u_1,\dots,u_n$ be a homogeneous $d$-sequence with $\deg(u_i)=d_i$
in $R$ such that $u_1,\dots,u_{n-1}$ is a regular sequence. Set $U
= (u_1, \ldots, u_n)$ and $d= \max \{d_i : 1 \leq i \leq n\}$.
Then, for all $s\geq 1$,
\[\reg ({R}/{U^s})\leq d(s-1) +\max \{ \reg({R}/{U}), \sum_{i=1}^{n-1}d_i - n  \}. \]
\end{corollary}
\begin{proof}
Immediately follows from Theorem
\ref{reg-graded-filtration} and Proposition \ref{reg-relative-sequence}.
\end{proof}
	
\section{Regularity of powers of binomial edge ideals}

While there is extensive research on the regularity of powers of
monomial edge ideals corresponding to graphs, there is absolutely no
such result in the case of binomial edge ideals. In this section, we
obtain bounds as well as precise expressions for the regularity of
powers of binomial edge ideals corresponding to some classes of
graphs. We first recall the terminologies that are needed from graph
theory for our purpose.

Let $G$  be a  simple graph with the vertex set $V(G)=[n]$ and edge
set $E(G)$. A \textit{complete graph} on $[n]$, denoted by $K_n$, is
the graph with the edge set $E(G)=\{ \{i,j\}: 1\leq i < j \leq n \}$.
For $A \subseteq V(G)$, $G[A]$ denotes the \textit{induced subgraph}
of $G$ on the vertex set $A$, that is, for $i, j \in A$, $\{i,j\} \in
E(G[A])$ if and only if $ \{i,j\} \in E(G)$.  For a vertex $v$, $G
\setminus v$ denotes the  induced subgraph of $G$ on the vertex set
$V(G) \setminus \{v\}$.  A subset $U$ of $V(G)$ is said to be a
\textit{clique} if $G[U]$ is a complete graph. A vertex $v$ of $G$ is
said to be a \textit{simplicial vertex} if $v$ is contained in only
one maximal clique otherwise it is called an \textit{internal vertex}.
We denote the number of internal vertices of $G$ by $\iv(G)$. For a
vertex $v$, $N_G(v) = \{u \in V(G) :  \{u,v\} \in E(G)\}$ denotes the
\textit{neighborhood} of $v$ in $G$ and  $G_v$ is the graph on the
vertex set $V(G)$ and edge set $E(G_v) =E(G) \cup \{ \{u,w\}: u,w \in
N_G(v)\}$.  The \textit{degree} of a vertex  $v$, denoted by
$\deg_G(v)$, is $|N_G(v)|$. A vertex $v$ is said to be a
\textit{pendant vertex} if $\deg_G(v) =1$.   For an edge $e$ in $G$,
$G\setminus e$ is the graph on the vertex set $V(G)$ and edge set
$E(G) \setminus \{e\}$.  Let $u,v \in V(G)$ be such that $e=\{u,v\}
\notin E(G)$, then we denote by $G_e$, the graph on the vertex set
$V(G)$ and edge set $E(G_e) = E(G) \cup \{\{x,y\} : x,\; y \in N_G(u)
\; or \; x,\; y \in N_G(v) \}$.  A \textit{cycle} is a connected graph
$G$ with $\deg_G(v) = 2$ for all $v \in V(G)$. A graph is said to be a
\textit{unicyclic} graph if it contains exactly one cycle. A graph is
a \textit{tree} if it does not contain a cycle. The length of a
shortest induced cycle in $G$ is called the \textit{girth} of $G$. A
\textit{path graph} on $n$ vertices, denoted by $P_n$, is a graph with
the vertex set $[n]$ and the edge set $\{ \{i,i+1\}: 1\leq i\leq n-1
\}$.  A vertex $v$ of $G$ is said to be a \textit{cut vertex} if $G
\setminus v$ has more connected components than $G$.  A \textit{block}
of a graph is a  maximal nontrivial connected subgraph which has no
cut vertex. If every block of a connected graph $G$ is a complete
graph, then $G$ is called a \textit{block graph}.  Let $G$ be a tree
and $L(G) = \{v \in V(G) : \deg_G(v) = 1\}$. If $G[L(G)^c]$ is either
empty or a path, then $G$ is said to be a \textit{caterpillar}. A
collection of edges $\{e_1, \ldots, e_s\}$ is said to be a
\textit{matching} if $e_i \cap e_j = \emptyset$ for all $i \neq j$ and
this is said to be an \textit{induced matching} if the induced
subgraph on the vertices of $\{e_1, \ldots, e_s\}$ has edge set
$\{e_1, \ldots, e_s\}$. For a graph $G$, the induced matching number
of $G$ is the largest size of an induced matching, and it is denoted
by $\nu(G)$.

\begin{notation}{\em
Let $G$ be a graph on $[n]$. We reserve the notation $S$ for the
polynomial ring $\K[x_i, y_i : i \in [n]]$. Also, if there is only
one graph on $[n]$ in a given context, irrespective of the
notation used for this graph, the polynomial ring associated with
this graph would be denoted by $S$. If $H$ is any other graph on
$[k]$, then we set $S_H = \K[x_i, y_i : i \in [k]]$.  If $k \leq
n$, then $S_H$ can be considered as a subring of $S$. Note that
the graded Betti numbers of $J_H$ considered as an ideal of $S_H$
are the same as those when considered as an ideal of $S$. Therefore,
for the convenience of notation, in such cases, we consider $J_H$
as an ideal in $S$. For an edge $e = \{i, j\}$ with $i < j$, let
$f_e = x_iy_j - x_jy_i$.  }
\end{notation}

We first make some observations, which can be derived easily 
from the results existing in the literature.

\begin{obs}\label{ttt}
\par  {\em (1) Let $G$ be a disjoint union of $k$ paths with $|V(G)| =
n$. Then, it follows from
\cite[Corollary 1.2]{EHHNMJ} that $J_G$ is generated by a
regular sequence in degree $2$ of length $n-k$. Therefore, by virtue of \cite[Lemma 4.4]{STT},
$\reg(S/J^s_{P_n})=2s+n-k-2$ for all $s \geq 1$.}
\par {\em (2) Let $G = K_n$ and $I = \ini_{\lex}(J_G)$, where $\lex$
denotes the lexicographic order on $S$ induced by
$x_1>\dots>x_n>y_1>\dots>y_n$. By \cite[Theorem 2.1]{Conca97},
$I^s = \ini_{\lex}(J_G^s)$. Hence, $\reg(S/J_G^s) \leq
\reg(S/\ini_{\lex}(J_G^s)) = \reg(S/I^s)$. Note that $I$ is a quadratic squarefree
monomial ideal. If $H$ denotes the graph corresponding to the
ideal $I$, then $H$ is a weakly chordal bipartite graph,
\cite[Lemma 3.3]{EZ}. Hence, by \cite[Corollary 5.1]{JNS},
$\reg(S/I^s)  = 2s + \nu(H) - 2$. It is easy to observe in this case
that $\nu(H) = 1$. Hence, $\reg(S/I^s) = 2s - 1$ which implies
that $\reg(S/J_G^s)
\leq 2s - 1$. Since $J_G^s$ is generated in degree $2s$, we get
$\reg(S/J_G^s) = 2s-1$. Therefore, $J_G^s$ has linear resolution for
all $s \geq 1$. For $s = 1$, this property has been proved by
Saeedi Madani and Kiani in \cite[Theorem 2.1]{KMEJC}.}
\end{obs}
	
Matsuda and Murai  \cite[Corollary 2.2]{MM}  proved that if $H$ is an
induced subgraph of $G$, then $\beta_{i,j}(S/J_H) \leq
\beta_{i,j}(S/J_G)$ for all $i,j$. We generalize this to all powers.
\begin{proposition}\label{induced-subgraph}
		Let $H$ be an induced subgraph of $G$. Then, for all $i, j\geq 0$ and
		$s \geq 1$ $\beta_{i,j}(S/J^s_H)\leq \beta_{i,j}(S/J^s_G)$.
	\end{proposition}
	\begin{proof}
		First we claim that $J_H^s = J_G^s \cap S_H$ for all $s \geq 1$, where
		$J_H$ is the binomial edge ideal of $H$ in  $S_H$. Since generators of
		$J_H^s$ are contained in $J_G^s$, $J_H^s \subseteq J_G^s \cap S_H$. Now,
		let $g \in J_G^s \cap S_H$.  Let $g = \sum_{ e_1, \ldots,e_s \in E(G)} h_{e_1,\ldots,e_s} f_{e_1}
		\cdots f_{e_s},$ where  $ h_{e_1,\ldots,e_s} \in S$. Now, consider
		the map  $\pi: S \to S_H$ by setting $\pi(x_i)= 0 = \pi(y_i)$ if $i
		\notin V(H)$ and $\pi(x_i) = x_i, \pi(y_i) = y_i$ if $i \in V(H)$. If 
		$e \in E(G) \setminus E(H)$,
		then $\pi(f_e) = 0$. If $e \in E(H)$, then $\pi(f_e) = f_e$. Since
		$g \in S_H$, $\pi(g) = g$.
		Therefore, we get 
		\begin{eqnarray*}
			g & = &\sum_{ e_1, \ldots,e_s \in E(G)}
			\pi(h_{e_1,\ldots,e_s}) \pi(f_{e_1}) \cdots \pi(f_{e_s}) \\
			& = & \sum_{e_1, \ldots,e_s \in E(H)}
			\pi(h_{e_1, \ldots,e_s}) f_{e_{1}} \cdots f_{e_{s}}.
		\end{eqnarray*}
		Thus, $g \in J_H^s$. Hence, $S_H/J_H^s$ is a $\K$-subalgebra of
		$S/J_G^s$. Consider, $S_H/J_H^s \xhookrightarrow{i} S/J_G^s
		\stackrel{\bar{\pi}}\rightarrow S_H/J_H^s$, where $\bar{\pi}$ is
		induced by the map $\pi$. 
		Note that $\bar{\pi} \circ i$ is identity on $S_H/J_H^s$. Thus,
		$S_H/J_H^s$ is an algebra retract of $S/J_G^s$. Now, the assertion follows
		from \cite[Corollary 2.5]{ohh2000}.
	\end{proof}
	\begin{corollary}\label{lower-bound}
		Let $G$ be a connected graph. Then, $\reg(S/J^s_G)\geq 2s+\ell(G)-2$ for all $s \geq 1$, where
		$\ell(G)$ is the length of a longest induced path of $G$.
	\end{corollary}
	\begin{proof}
		Let $H$ be a longest induced path of $G$. Then, $H$ is an induced
		subgraph of $G$. By Observation \ref{ttt}, $\reg(S_H/J_H^s)
		=2s +\ell(G)-2$ for all $s \geq 1$. Hence, the assertion follows from
		Proposition \ref{induced-subgraph}.
	\end{proof}
	
	In \cite{JAR1}, we had shown that $J_{K_{1,n}}$ is generated by a
	$d$-sequence. Schenzel and Zafar proved that $\reg(S/J_{K_{1,n}}) = 2$,
	\cite{Schenzel}. We now compute the regularity of their powers. 
	\begin{theorem}
		Let $G=K_{1,n}$ be a star graph for $n\geq 3$. Then, $\reg(S/J_G^s)=2s$ for all $s\geq 1$.
	\end{theorem}
	\begin{proof}
		Let $G=K_{1,n}$ denote the star graph on the vertex set $[n+1]$ with
		the edge set $E(G)=\{ \{i,n+1\}:1\leq i\leq n\}.$ By \cite[Theorems
		4.1]{Schenzel}, $\reg(S/J_G) =2$. It follows from
		\cite[Proposition 4.8]{JAR1} that $J_G=(f_{1,n+1},\ldots,f_{n,n+1})$
		is generated by a quadratic sequence with
		respect to the poset $\Lambda=\{\{1,n+1\}<\cdots <\{n,n+1\} \}$. Let $V$ be a related ideal to $J_G$. Then
		$V$ is either $J_G$ or of the form
		$((f_{1,n+1},\dots,f_{i,n+1}):f_{i+1,n+1})+J_G$ for some $i\geq 2$. If
		$V=((f_{1,n+1},\dots,f_{i,n+1}):f_{i+1,n+1})+J_G$ for some $i \geq 2$, then by
		\cite[Theorem 3.7]{FM}, $V=J_H$, where $H$ is the graph obtained from
		the complete graph on vertex set $\{1,\dots,i,n+1\}$ by adding
		edges  $\{j,n+1\}$ for $i+1\leq j \leq n$ at the vertex $n+1$.
		Thus, by using \cite[Theorem 8]{her2}, $\reg(S/V)=\reg(S/J_H)\leq 2$.
		Therefore, by Theorem \ref{reg-graded-filtration}, we have
		$\reg(S/J_G^s)\leq 2(s-1)+2=2s$ for all $s \geq 1$. Since $\ell(G)=2$, by Corollary \ref{lower-bound},  $\reg(S/J_G^s)\geq 2s$ for all $s \geq 1$. Hence, for all $s \geq 1$, $\reg(S/J_G^s)=2s$.
	\end{proof}
	Now, we obtain the regularity of powers of binomial edge ideals of cycle graphs. 
	\begin{theorem}
		Let $n \geq 3$. Then, for all $s \geq 1$, $\reg(S/J_{C_n}^s)=2s+n-4.$
	\end{theorem}
	\begin{proof}
		Let $G=C_n$ be the cycle graph on $[n]$. By \cite[Corollary
		16]{Zafar}, we have $\reg(S/J_G)=n-2$. Moreover,
		$f_{1,2},\dots,f_{n-1,n}$ is a regular sequence and by \cite[Theorem
		4.9]{JAR1}, $f_{1,2},\ldots,f_{n-1,n},f_{1,n}$ is a $d$-sequence.
		Hence, it follows from Corollary
		\ref{reg-d-sequence} that $\reg(S/J^s_G)\leq 2s+n-4$ for all $s\geq 1$. Since $C_n$ contains
		an induced path of length $n-2$, it follows from Corollary
		\ref{lower-bound} that $2s+n-4 \leq \reg(S/J^s_G)$ for all $s\geq 1$. Hence, $\reg(S/J^s_G)=2s+n-4$ for all $s\geq 1$.
	\end{proof}
	
	In \cite{JNR}, it was proved that if $G$ is a tree, then $\iv(G)+1 \leq
	\reg(S/J_G)$. In the next result, we obtain a lower bound
	for all powers of almost complete intersection binomial
	edge ideals of trees. We also give an upper bound for
	this class.
	\begin{theorem}\label{aci-tree}
		If $G$ is a tree such that $J_G$ is an almost complete intersection
		ideal, then for all $s \geq 1$, $2s + \iv(G) - 2 \leq \reg(S/J_G^s)
		\leq 2s + \iv(G) - 1.$
	\end{theorem}
	
	\begin{proof}
		Let $G$ be a tree such that
		$J_G$ is an almost complete intersection ideal.  Then, by \cite[Theorem
		4.1]{JAR1}, $G$ is obtained by adding an edge between two paths,
		either by adding an edge between two internal vertices or by adding an edge between an internal vertex and a pendant vertex. If $G$ is
		obtained by adding an edge between an internal vertex of a path and a
		pendant vertex of another path, then we say that $G$ is type $T$ and
		if $G$ is obtained by adding an edge between two internal vertices of
		two distinct paths, then we say that $G$ is type $H$. Note that if $G$
		is $T$-type, then $\iv(G) = n-3$ and if $G$ is $H$-type, then $\iv(G) =
		n-4$.
		
		Let $G$ be $T$-type. Then, by \cite[Theorems 4.1, 4.2]{JNR}, $G$
		contains no Jewel graph as an induced subgraph, and hence,
		$\reg(S/J_G)=n-2$. Therefore, it follows from Corollary
		\ref{reg-d-sequence} that $\reg(S/J_G^s)\leq 2s+n-4 = 2s + \iv(G) - 1$
		for all $s \geq 1$. Let $v$ denote a neighbor of the unique vertex of
		degree $3$ in $G$. Then, $J_{G \setminus v}$ is generated by a regular
		sequence of length at least $n-3$. Hence, for all $s \geq 1$, $2s + n- 5
		\leq \reg(S/(J_{G\setminus v})^s) \leq \reg(S/J_G^s)$, where the first inequality
		follows from \cite[Lemma 4.4]{STT} and the second inequality follows from
		Proposition \ref{induced-subgraph}. Therefore, $2s + \iv(G) -2 \leq
		\reg(S/J_G^s)$ for all $s \geq 1$.

		Now, let $G$ be $H$-type
		and $e=\{u,v\}$ be the edge such that
		$G\setminus e$ is a disjoint union of two paths. Then, the related ideal
		to $J_G$ is either $J_G$ or of the form $(J_{G\setminus e} : f_e) +
		J_G$. By \cite[Theorems 4.1, 4.2]{JNR},
		$\reg(S/J_G)=n-3$.  It follows from \cite[Theorem 3.7]{FM} that $(J_{G\setminus e} : f_e) +
		J_G = J_{(G\setminus e)_e \cup \{e\}}$. Since $(G\setminus e)_e \cup
		\{e\}$ is a block graph with no vertex contained in more than two
		maximal cliques, it follows from \cite[Corollary 3.1]{JNR} that
		$\reg(S/J_{(G\setminus e)_e \cup \{e\}})=n-3 = \iv(G) + 1$. Hence, by
		Theorem \ref{reg-graded-filtration}, $\reg(S/J_G^s)\leq
		2s+n-5 = 2s + \iv(G) - 1$. As in the previous case, it can be seen
		that $J_{G \setminus u}$ is generated by a regular sequence of length
		$n-4$. Hence, by \cite[Lemma 4.4]{STT} and Proposition
		\ref{induced-subgraph}, we get $2s + \iv(G) - 2 = 2s + n-6 =
		\reg(S/J_{G \setminus u}^s) \leq \reg(S/J_G^s)$ for all $s \geq 1$.
	\end{proof}

	\begin{corollary}\label{caterpillar}
		If $G$ is a caterpillar tree such that $J_G$ is an almost complete
		intersection, then $\reg(S/J_G^s) = 2s + \iv(G) - 1$.
	\end{corollary}
	
	\begin{proof}
		If $G$ is a caterpillar tree on $[n]$, then $G$ has a longest induced
		path, say $P$, such that $\iv(G) = \iv(P)=\ell(P)-1$. Hence, by
		Corollary \ref{lower-bound}, $2s + \iv(G) - 1 \leq
		\reg(S/J_G^s)$. The upper bound follows from Theorem \ref{aci-tree}.
	\end{proof}
	
	We are unable to prove, but believe that the answer to the following
	question is affirmative:
	
\begin{question}
If $G$ is a tree, then is $2s + \iv(G) - 2 \leq \reg(S/J_G^s)$ for all
$s \geq 1$?
\end{question}

Now, we deal with unicyclic graphs, other than cycles, whose
binomial edge ideals are almost complete intersection. We first
develop the tools required for that.  
	\begin{definition}
		Let $G_1$ and $G_2$ be two subgraphs of a graph $G$. If $G_1\cap
		G_2=K_m$, the complete graph on $m$ vertices with $G_1\neq K_m$ and
		$G_2\neq K_m$, then $G$ is called the clique sum of $G_1$ and $G_2$
		along the complete graph $K_m$, denoted by $G_1\cup_{K_m} G_2$. If
		$m=1$, the clique sum of $G_1$ and $G_2$ along a vertex is denoted by
		$G_1\cup G_2$. If $m =2$ and $K_2 = e$, then 	the clique sum of
		$G_1$ and $G_2$ along $e$ is denoted by $G_1 \cup_e G_2$.
	\end{definition}
	
	\begin{minipage}\linewidth
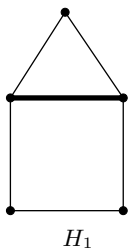
		\begin{minipage}{0.2\linewidth}
			\begin{tikzpicture}[scale=0.75]
			\draw [line width=0.5pt] (1,3)-- (1,1);
			\draw [line width=0.5pt] (1,1)-- (3,1);
			\draw [line width=0.5pt] (3,1)-- (3,3);
			\draw [line width=2pt] (3,3)-- (1,3);
			\draw [line width=0.5pt] (1,3)-- (1.97,4.52);
			\draw [line width=0.5pt] (1.97,4.52)-- (3,3);
			\begin{scriptsize}
			\draw [fill=black] (1,3) circle (2pt);
			\draw [fill=black] (1,1) circle (2pt);
			\draw [fill=black] (3,1) circle (2pt);
			\draw [fill=black] (3,3) circle (2pt);
			\draw [fill=black] (1.97,4.52) circle (2pt);
			\draw[color=black] (2.2, 0.5) node {$H_1$};
			\end{scriptsize}
			\end{tikzpicture}
		\end{minipage}
		\begin{minipage}{0.30\linewidth}
			
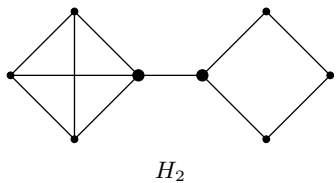
\begin{tikzpicture}[scale=0.85]
			\draw [line width=0.5pt] (2,3)-- (1,2);
			\draw [line width=0.5pt] (1,2)-- (2,1);
			\draw [line width=0.5pt] (2,1)-- (3,2);
			\draw [line width=0.5pt] (3,2)-- (2,3);
			\draw [line width=0.5pt] (2,3)-- (2,1);
			\draw [line width=0.5pt] (1,2)-- (3,2);
			\draw [line width=0.5pt] (3,2)-- (4,2);
			\draw [line width=0.5pt] (4,2)-- (5,3);
			\draw [line width=0.5pt] (5,3)-- (6,2);
			\draw [line width=0.5pt] (6,2)-- (5,1);
			\draw [line width=0.5pt] (5,1)-- (4,2);
			\begin{scriptsize}
			\draw [fill=black] (2,3) circle (1.5pt);
			\draw [fill=black] (1,2) circle (1.5pt);
			\draw [fill=black] (2,1) circle (1.5pt);
			\draw [fill=black] (3,2) circle (2.5pt);
			\draw [fill=black] (4,2) circle (2.5pt);
			\draw [fill=black] (5,3) circle (1.5pt);
			\draw [fill=black] (6,2) circle (1.5pt);
			\draw [fill=black] (5,1) circle (1.5pt);
			\draw[color=black] (3.5, 0.5) node {$H_2$};
			\end{scriptsize}
			\end{tikzpicture}
		\end{minipage}
		\begin{minipage}{0.35\linewidth}
			The first graph $H_1$ on the left is the clique sum of a $K_3$ and
			$C_4$ along an edge. The second graph, $H_2$, is a clique sum of a $K_4$, an
			edge and a $C_4$ along two vertices.
		\end{minipage}
	\end{minipage}

	To understand the regularity of powers of binomial edge ideals of
	unicyclic graphs which are almost complete intersections, we first prove
	an auxiliary result.
	\begin{proposition}\label{clique-sum-regularity}
		Let $n, m \geq 3$ be integers and $G$ be the clique sum
		of a cycle $C_n$ and a complete graph $K_m$ along an edge $e$.
		Then, $\reg(S/J_G)=n-1$.
	\end{proposition}
	\begin{proof}
		It is easy to notice that $G$ contains an induced path of length $n-1$.
		Hence, by \cite[Corollary 2.3]{MM}, $n-1\leq \reg (S/J_G)$. 
		Note that $G$ is a graph on $n+m-2$ vertices. Since $K_m$ is the
		maximal clique of largest size in $G$, By \cite[Theorem 2.1]{ERT20},  
		$\reg(S/J_G) \leq (n+m-2) - (m-1) = n-1$.
	\end{proof}
	\noindent
	\begin{minipage}\linewidth
		\begin{minipage}{0.55\linewidth}
			Let $G_1$ and $G_2$ denote 
			graphs on the vertex set $[m]$ with edge sets given by $E(G_1)
			=\{\{1,2\},\{2,3\},\ldots, \{m-1,m\},\{2,m\}\}$ and
			$E(G_2)=\{\{1,2\},\{2,3\},\ldots, \{m-1,m\},\{2,m-1\}\}$. Let
			$e=\{1,2\},e'=\{m-1,m\}$ denote edges in $G_2$.
		\end{minipage}
		\begin{minipage}{0.45\linewidth}
			\hspace*{8mm}
			\captionsetup[figure]{labelformat=empty}
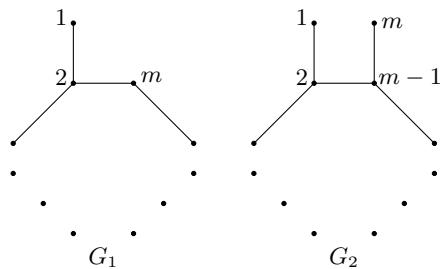
			\begin{figure}[H]\label{fig1}
				\begin{tikzpicture}[scale=0.8]
				\draw  (2,4)-- (2,3);
				\draw  (3,3)-- (4,2);
				\draw  (2,3)-- (1,2);
				\draw  (5,2)-- (6,3);
				\draw  (6,3)-- (6,4);
				\draw  (2,3)-- (3,3);
				\draw  (6,3)-- (7,3);
				\draw  (7,4)-- (7,3);
				\draw  (7,3)-- (8,2);
				\begin{scriptsize}
				\draw [fill=black] (2,4) circle (1pt);
				\draw [fill=black] (2,3) circle (1pt);
				\draw [fill=black] (7,4) circle (1pt);
				\draw [fill=black] (3,3) circle (1pt);
				\draw [fill=black] (4,2) circle (1pt);
				\draw [fill=black] (1,2) circle (1pt);
				\draw [fill=black] (5,2) circle (1pt);
				\draw [fill=black] (6,3) circle (1pt);
				\draw [fill=black] (6,4) circle (1pt);
				\draw [fill=black] (7,3) circle (1pt);
				\draw [fill=black] (8,2) circle (1pt);
				\draw [fill=black] (1,1.5) circle (1pt);
				\draw [fill=black] (1.5,1) circle (1pt);
				\draw [fill=black] (2,0.5) circle (1pt);
				\draw [fill=black] (3,0.5) circle (1pt);
				\draw [fill=black] (3.5,1) circle (1pt);
				\draw [fill=black] (4,1.5) circle (1pt);
				\draw [fill=black] (5,1.5) circle (1pt);
				\draw [fill=black] (5.5,1) circle (1pt);
				\draw [fill=black] (6,0.5) circle (1pt);
				\draw [fill=black] (7,0.5) circle (1pt);
				\draw [fill=black] (7.5,1) circle (1pt);
				\draw [fill=black] (8,1.5) circle (1pt);
				\draw (7.6, 3.1) node {$m-1$};
				\draw (7.3, 4) node {$m$};
				\draw (5.8, 3.1) node {$2$};
				\draw (5.8, 4.1) node {$1$};
				\draw (1.8, 4.1) node {$1$};
				\draw (1.8, 3.1) node {$2$};
				\draw (3.3, 3.1) node {$m$};
				\draw (2.5, 0.1) node {$G_1$};
				\draw (6.5, 0.1) node {$G_2$};
				\end{scriptsize}
				\end{tikzpicture}
			\end{figure}
		\end{minipage}
	\end{minipage}
	
	Now, we proceed to study the regularity of powers of almost complete intersection binomial edge ideals of unicyclic graphs.
	In \cite{JAR1}, we had proved that any such graph is obtained by
	attaching paths to the free vertices of $G_1$ and $G_2$. So, we first
	compute the regularity of powers of $J_{G_1}$ and $J_{G_2}$.
	
	\begin{proposition}\label{g1}
		Let $m\geq 4$ and $G_1$ be the graph as given above. Then, $\reg(S/J_{G_1})=m-2$.
	\end{proposition}
	\begin{proof}
		By \cite[Theorem 3.2]{KMJCTA}, $\reg(S/J_{G_1})\leq m-2$. Note that $G_1\setminus m$ is an induced
		path in $G_1$ of length $m-2$. So by
		\cite[Corollary 2.3]{MM}, $m-2\leq
		\reg(S/J_{G_1})$. Hence, $\reg(S/J_{G_1})=m-2$.
	\end{proof}
	\begin{proposition}\label{g2}
		Let $G_2$ be the graph as given above for $m\geq 6$. Then, $\reg(S/J_{G_2})=m-3$.
	\end{proposition}
	\begin{proof}
		Note that $G_2\setminus(m-1)$ is an induced
		path in $G_2$ of length $m-3$. So by
		\cite[Corollary 2.3]{MM}, $m-3\leq \reg(S/J_{G_2})$. Now, we prove that $\reg(S/J_{G_2})\leq m-3$.
		Consider the following short exact
		sequences:
		\begin{align}\label{1st}
			0\longrightarrow S/(J_{G_2\setminus e}:f_e)(-2)\stackrel{\cdot f_e}{\longrightarrow}S/J_{G_2\setminus e}\longrightarrow S/J_{G_2}\longrightarrow 0
		\end{align}
		\begin{align}\label{2nd}
			0\longrightarrow S/(J_{H\setminus e'}:f_{e'})(-2)\stackrel{\cdot f_{e'}}{\longrightarrow}S/J_{H\setminus e'}\longrightarrow S/J_{H}\longrightarrow 0
		\end{align}
		where $H=(G_2\setminus e)_{e}$, $e=\{1,2\}$ and $e'=\{m-1,m\}$. Since
		$G_2 \setminus e$ is a graph isomorphic to $G_1$ on $m-1$ vertices,
		we get $\reg(S/J_{G_2 \setminus e}) = m-3$. By \cite[Theorem 3.7]{FM}, $J_{G_2\setminus e}:f_e=J_{(G_2\setminus e)_{e}}$ and $J_{H\setminus e'}:f_{e'}=J_{(H\setminus e')_{e'}}$.
		Note that $H\setminus e'=C_{m-3}\cup_{\{3,m-1\}}K_3$ and $(H\setminus
		e')_{e'}=C_{m-4}\cup_{\{3,m-2\}}K_4$. Thus, by using Proposition
		\ref{clique-sum-regularity}, $\reg(S/J_{H\setminus e'})=m-4$ and
		$\reg(S/J_{(H\setminus e')_{e'}})=m-5$. Now, applying Lemma
		\ref{regularity-lemma} on short exact sequences (\ref{1st}) and
		(\ref{2nd}), we get $\reg(S/J_H)\leq m-4$, and hence, $\reg(S/J_{G_2})\leq m-3$.
	\end{proof}
	
Now, we prove the bounds for the regularity of powers of binomial edge
ideals of unicyclic graphs which are almost complete intersections.
\begin{theorem} \label{aci-uc-bound}
Let $G$ be a unicyclic graph on $[n]$ which is not a cycle such that
$J_G$ is an almost complete intersection ideal. Assume that $\K$
is an infinite field. Then, $2s + n - 5 \leq \reg(S/J_G^s) \leq
2s+n-4$.
\end{theorem}

\begin{proof}
If $G$ is a unicyclic graph on $[n]$ such that $J_G$ is an almost
complete intersection ideal, then it was shown in \cite{JAR1} that
$G$ is obtained either by identifying a pendant vertex of a path
with the vertex $1$ in $G_1$ or by identifying a pendant vertex
each of two paths with the vertices $1$  and $m$ in $G_2$ or
attaching  paths to every vertex of $K_3$. If $u$ is a degree $3$
vertex of $G$, then $G \setminus u$ is a disjoint union of two
paths so that $J_{G \setminus u}$ is generated by a regular
sequence of length $n-3$.  Hence, by \cite[Lemma 4.4]{STT} and
Proposition \ref{induced-subgraph}, $2s+n-5 \leq \reg(S/J_G^s)$.
Since $J_G$ is an almost complete intersection ideal, it follows
from \cite[Proposition 5.1]{GM99}  that there exists a set of
homogeneous generators $\{F_1,\ldots,F_{n}\}$ of $J_G$ such that
$F_1,\ldots, F_{n-1}$ is a regular sequence in $S$.  Since
$J=(F_1,\ldots,F_{n-1})$ is an unmixed ideal, by  \cite[Theorem
4.7]{HMV89}, $J:F_{n} =J:F_{n}^2$. Hence, $J_G$ is generated by a
homogeneous $d$-sequence $F_1,\ldots,F_{n}$. Since $J_G$ is
generated in degree $2$, $F_i$ has degree $2$ for each $i$. Now,
the upper bound follows directly from Corollaries
\ref{reg-d-sequence} and \cite[Theorem 3.2]{KMJCTA}.
\end{proof}
\begin{remark}{\em
Suppose $G$ is a balloon graph on $[n]$, i.e., $G$ is obtained by
identifying a pendant vertex  of a path with the pendant vertex $G_1$.
Let $v$ denote a neighbor on the cycle of the degree $3$ vertex in
$G$. Then, $G \setminus v$ is a path of length $n-2$. Hence, by
Proposition \ref{induced-subgraph}, $2s + n - 4 \leq \reg(S/J_G^s)$
for all $s \geq 1$. Therefore, by Theorem \ref{aci-uc-bound},
$\reg(S/J_G^s) = 2s + n - 4$.
}
\end{remark}

\begin{remark}{\em
If $G$ is a unicyclic graph obtained by identifying a pendant vertex
of two distinct paths
to two pendant vertices of the graph $G_2$, then the regularity behaves
in an unexpected manner. For example, if we take $G=G_2$ on $6$
vertices, then $G$ is a Cohen-Macaulay bipartite graph, \cite{dav}.
Hence, $\reg(S/J_G) = 3$, \cite{JA1}. Using any computational
commutative algebra package, for example Macaulay 2 \cite{M2}, it can
be seen that $\reg(S/J_G^2) = 6 = \reg(S/J_G) + 3$. This is different
behavior in comparison with the regularity of powers of monomial edge
ideals.  It has been conjectured and is believed to be true, that
$\reg(I(G)^s) \leq 2s + \reg(I(G)) - 2$ for all $s \geq 1$, where
$I(G)$ denotes the monomial edge ideal corresponding to a graph $G$,
\cite{BBH17}. This example shows that such an inequality does not hold
in the case of binomial edge ideals. Moreover, this gives a
binomial edge ideal for which the stabilization index is bigger than
$1$.
}
\end{remark}
\section{Regularity of powers of parity binomial edge ideals}

In this section, we study the regularity of powers of parity binomial
edge ideals of some classes of graphs. It was proved by Bolognini et
al., that for bipartite graphs, the parity binomial edge ideals are
essentially the same as the binomial edge ideals, \cite[Corollary
6.2]{dav}. This implies that the algebraic invariants associated to
both these ideals are the same for bipartite graphs.  Hence, we
restrict our attention to graphs containing odd cycles.  In
\cite{AR5}, Kumar computed the regularity of parity binomial edge
ideals of odd cycles. We compute the regularity of their powers. 
\begin{theorem}\label{p-cycle}
Let $n \geq 3$ be an odd integer. Then, for all $s \geq 1$,
$\reg(S/\mathcal{I}_{C_n}^s)=2s+n-2.$
\end{theorem}
\begin{proof}
It follows from \cite[Theorem 3.5]{AR4} that $\mathcal{I}_{C_n}$ is
generated by a regular sequence of length $n$ in degree $2$.
Therefore, by virtue of \cite[Lemma 4.4]{STT},
$\reg(S/\mathcal{I}^s_{C_n})=2s+n-2$ for all $s \geq 1$.
\end{proof}

If $H$ is an induced subgraph of $G$, then
$\beta_{i,j}(S/\mathcal{I}_H) \leq \beta_{i,j}(S/\mathcal{I}_G)$ for
all $i,j$, \cite{AR5}. We generalize this to all powers. For an edge
$\{i,j\}, i < j$, let $\bar{g}_e = x_ix_j - y_iy_j$.

\begin{proposition}\label{p-betti-induced}
Let $G$ be a graph and $H$ be an induced subgraph of $G$. Then
$\beta_{i,j}^{S_H}(S_H/\mathcal{I}_H^s) \leq
\beta_{i,j}^S(S/\mathcal{I}_G^s)$, for all $i,j$, where
$S_H=\K[x_k,y_k : k \in V(H)]$.
\end{proposition}
\begin{proof}
The proof is essentially the same as the proof of Proposition
\ref{induced-subgraph}. One only needs to replace $f_{e_i}$ by $\bar{g}_{e_i}$.
\end{proof}
	
Let $\oc(G)$ denote the length of a longest induced odd cycle.  If $G$
has no induced odd cycle, then we assume that $\oc(G)=0$. In \cite{AR5},
Kumar proved that  $\reg(S/\mathcal{I}_G) \geq
\max\{\ell(G),\oc(G)\}$. We generalize this to all powers.
	\begin{corollary}\label{p-lower-bound}
		Let $G$ be a connected graph. Then, $\reg(S/\mathcal{I}^s_G)\geq 2s+\max\{\ell(G),\oc(G)\}-2$ for all $s \geq 1$.
	\end{corollary}
	\begin{proof}
		Without loss of generality, we assume that $G$ is a non-bipartite
		graph.	Let $H$ be a longest induced path of $G$. Then, $H$ is an
		induced subgraph of $G$. Since, $H$ is a bipartite graph, by
		\cite[Corollary 6.2]{dav} and Observation \ref{ttt},
		$\reg(S_H/\mathcal{I}_H^s) =2s +\ell(G)-2$ for all $s \geq 1$. Now,
		let $H'$ be a longest induced odd cycle. By Theorem \ref{p-cycle},
		$\reg(S_{H'}/\mathcal{I}_{H'}^s) = 2s + \oc(H') - 2 = 2s + \oc(G) -
		2$ for all $s \geq 1$. Hence, the
		assertion follows from Proposition \ref{p-betti-induced}.
	\end{proof}
	For the rest of the section, we assume that  $\chara(\K) \neq 2$. In
	the next three results, we study the regularity of powers of parity
	binomial edge ideals which are generated by $d$-sequence. 
	\begin{theorem}
		Let $G$ be a  graph on $[n]$ obtained by adding an edge between an odd
		cycle and an internal vertex of a path. Then, $2s+n-5 \leq \reg(S/\mathcal{I}_G^s) \leq 2s+n-4$, for all $s \geq 1$. 
	\end{theorem}
	\begin{proof}
		Let $G$ be obtained by adding edge $e = \{u, v\}$ between an odd cycle
		and a path, where $u$ lies on
		an odd cycle and $v$ is an internal vertex of a path.
		Note that $\mathcal{I}_G=\mathcal{I}_{G\setminus e} +(\bar{g}_e)$. By
		\cite[Corollary 3.6]{AR4}, $\mathcal{I}_{G\setminus e}$ is a complete
		intersection ideal. It follows from the proof of  \cite[Theorem
		3.8(1)]{AR4} that $\mathcal{I}_G$ is generated by a $d$-sequence of
		length $n$ such that first $n-1$ elements form a regular sequence.
		Thus, by Corollary \ref{reg-d-sequence}, \[\reg(S/\mathcal{I}_G^s)\leq
		2(s-1)+\max\{\reg(S/\mathcal{I}_G),n-2\}, \text{ for all } s \geq
		1.\] Note that there exists an edge $e'$ such that $G\setminus e'$
		is a bipartite graph and $e' \cap \{u,v\} = \emptyset$. Consider the
		short exact sequence 
		\begin{equation}\label{ses-ls}
			0 \longrightarrow \frac{S}{\mathcal{I}_{G\setminus
					{e'}}:\bar{g}_{e'}}(-2) \stackrel{\cdot \bar{g}_{e'}}\longrightarrow
			\frac{S}{\mathcal{I}_{G\setminus {e'}}} \longrightarrow
			\frac{S}{\mathcal{I}_G} \longrightarrow 0.\end{equation} It follows
		from \cite[Lemma 3.3]{AR4} that $\mathcal{I}_{G\setminus
			{e'}}:\bar{g}_{e'}=\mathcal{I}_{G\setminus {e'}}$. Hence, by Lemma
		\ref{regularity-lemma}, $\reg(S/ \mathcal{I}_G)
		=\reg(S/\mathcal{I}_{G\setminus {e'}})+1$. Note that $G \setminus
		{e'}$ is a tree of $H$-type. By \cite[Corollary 6.2]{dav} and the 
		proof of Theorem \ref{aci-tree},
		$\reg(S/\mathcal{I}_{G \setminus e'}) = n-3$ so that
		$\reg(S/\mathcal{I}_G) = n-2$.
		Hence, $\reg(S/\mathcal{I}_G^s) \leq 2s+n-4$, for all $s \geq 1$. Note
		that $G \setminus u$ is disjoint union of two paths on $n-1$ vertices.
		Therefore, $\reg(S/\mathcal{I}_{G \setminus u}^s)=2s + n-5 $, for all
		$s \geq 1$. Thus, by Proposition \ref{p-betti-induced}, $2s+n-5 \leq \reg(S/\mathcal{I}_G^s)$, for all $s \geq 1$. 
	\end{proof}
	\begin{theorem}
		Let $G$ be a balloon graph on $[n]$ having odd girth. Then, for all $s
		\geq 1$ $2s+n-4 \leq  \reg(S/\mathcal{I}_G^s) \leq 2s+n-3$. 
	\end{theorem}
	\begin{proof}
		Since $\ell(G) = n-1$, it follows from Corollary \ref{p-lower-bound}
		that $2s+n-4 \leq \reg(S/\mathcal{I}_G)$ for all $s \geq 1$.
		Let $u$ be the vertex of degree three in $G$ and $v$ be a neighbor of
		$u$ on the cycle. Set $e =\{u,v\}$. Then, $G \setminus e$ is a path graph. It
		follows from the proof of  \cite[Theorem 3.8(1)]{AR4} that $\mathcal{I}_G$
		is generated by a $d$-sequence of length $n$ such that first $n-1$
		elements form a regular sequence. Thus, by Corollary
		\ref{reg-d-sequence}, \[\reg(S/\mathcal{I}_G^s)\leq
		2(s-1)+\max\{\reg(S/\mathcal{I}_G),n-2\}, \text{ for all } s \geq
		1.\] If one takes $e'$ to be any edge on the cycle such that $e' \cap
		\{u\} = \emptyset$, then $G\setminus e'$ is a bipartite graph, and hence, it follows from \cite[Lemma 3.3]{AR4} that
		$\mathcal{I}_{G\setminus {e'}}:\bar{g}_{e'}=\mathcal{I}_{G\setminus
			{e'}}$. Thus, by Lemma \ref{regularity-lemma}, $\reg(S/
		\mathcal{I}_G) =\reg(S/\mathcal{I}_{G\setminus {e'}})+1$. Note that
		$G \setminus {e'}$ is a tree of $T$-type. 
		Therefore, by \cite[Corollary 6.2]{dav} and the proof of Theorem
		\ref{aci-tree}, $\reg(S/\mathcal{I}_{G \setminus e'}) = n-2$ so
		that $\reg(S/\mathcal{I}_G) = n-1$. Hence, $\reg(S/\mathcal{I}_G^s)
		\leq 2s+n-3$, for all $s \geq 1$.
	\end{proof}
	
	\begin{theorem}\label{p-chord}
		Let $G$ be a graph obtained by adding a chord in  an odd cycle $C_n$. Then, $2s+n-4 \leq \reg(S/\mathcal{I}_G^s)\leq 2s+n-3$, for all $s \geq 1$. 
	\end{theorem}
	\begin{proof}
		Since $\ell(G) = n-1$, it follows from Corollary \ref{p-lower-bound}
		that $2s+n-4 \leq \reg(S/\mathcal{I}_G)$ for all $s \geq 1$.
		Let $e = \{u, v\}$ be the chord in $C_n$.  Then, $\mathcal{I}_{G
			\setminus e}$ is a
		complete intersection. Moreover, $\mathcal{I}_{G \setminus e} :
		\bar{g}_e^2 = \mathcal{I}_{G \setminus e} : \bar{g}_e$. Hence,	$\mathcal{I}_G$ is generated by a $d$-sequence of length $n+1$.
		Therefore, by Corollary \ref{reg-d-sequence},
		$\reg(S/\mathcal{I}_G^s)\geq
		2(s-1)+\max\{\reg(S/\mathcal{I}_G),n-1\}$. To complete the proof, it is enough to show that $\reg(S/\mathcal{I}_G) \leq n-1$.
		The chord $e$ splits the graph $G$ into two induced cycles, an odd
		cycle and an even cycle, whose intersection is $e$.
		Let $e'=\{v,w\} \in E(G)$ be an edge of the induced odd cycle in $G$. 
		Observe that $G\setminus e'$ is a balloon graph having even
		girth. Since $G \setminus e'$ is a bipartite graph and for a bipartite
		graph the parity
		binomial edge ideal is isomorphic to the binomial edge ideal, we
		conclude using \cite[Theorem 3.2]{KMJCTA} that
		$\reg(S/\mathcal{I}_{G\setminus e'}) \leq n-2$.  Also, from
		\cite[Lemma 3.3]{AR4} we get $\mathcal{I}_{G \setminus e'}
		:\bar{g}_{e'}\simeq J_{(G\setminus e')_{e'}}$. Notice that
		$(G\setminus e')_{e'} =(G\setminus e')_v$ is not a path graph.
		Therefore, by \cite{KMJCTA}, $\reg(S/J_{(G\setminus e')_v}) \leq n-2$.
		Hence, the assertion follows  by applying Lemma \ref{regularity-lemma},
		on the short exact sequence \eqref{ses-ls}.
	\end{proof}
	
	\bibliographystyle{plain}  %% or 
	\bibliography{Reference}
	
\end{document}